\documentclass[10pt]{amsart} 
\pdfoutput=1 

\usepackage{ifpdf}

\usepackage{amsmath}
\usepackage{amsthm}
\usepackage{amsfonts}
\usepackage{amssymb}

\ifpdf
    \usepackage[colorlinks]{hyperref}
    \hypersetup{urlcolor=blue,linkcolor=blue,citecolor=blue}
\fi
\usepackage{amscd}	
\usepackage{cancel}	
\usepackage{microtype} 

\usepackage{graphicx}

\usepackage{color}

\ifdefined\draftversion
    \providecommand{\showkeys}{}
    \usepackage[initials,backrefs,msc-links]{amsrefs}
\else
    \usepackage[initials,msc-links]{amsrefs}
\fi

\ifdefined\showkeys
	\usepackage[color]{showkeys} 
\fi

\makeatletter
\def\blfootnote{\xdef\@thefnmark{}\@footnotetext} 
\makeatother

\IfFileExists{.git/git.tex}{
\input{.git/git.tex}
\newcommand{\gitversionfootnote}{\blfootnote{Version: \gitauthsdate, \gitshash}}

}
{
\newcommand{\gitversionfootnote}{}
}

\graphicspath{{}}
\ifpdf
  \newcommand{\fig}[2]{
    \IfFileExists{#1.pdf_tex}{
      \def\svgwidth{#2}\input{#1.pdf_tex}
    }{
      \frame{Missing figure ``#1.pdf''}
      \message{LaTeX Warning: Missing figure ``#1.pdf'' on input line \the\inputlineno}
    }
  }
\else
  \newcommand{\fig}[2]{\frame{PDF Figure here}}	
\fi

\newcommand{\dimension}{n}

\DeclareMathOperator{\dist}{dist}

\DeclareMathOperator{\interior}{int}
\DeclareMathOperator{\trace}{trace}

\DeclareMathOperator*{\esssup}{ess\,sup}
\DeclareMathOperator*{\essinf}{ess\,inf}

\DeclareMathOperator*{\argmax}{arg\,max}

\DeclareMathOperator{\divo}{div}

\newcommand{\diff}[1]{\;d{#1}}

\newcommand{\ov}[1]{\frac{1}{#1}}

\newcommand{\abs}[1]{\left|#1\right|}
\newcommand{\pth}[1]{\left(#1\right)}
\newcommand{\bra}[1]{\left[#1\right]}
\newcommand{\set}[1]{{\left\{#1\right\}}}

\newcommand{\at}[2]{{{\left.{#1}\right|}_{#2}}}
\newcommand{\ang}[1]{{\left\langle#1\right\rangle}}
\newcommand{\norm}[1]{\left\|#1\right\|}

\newcommand{\cl}[1]{\overline{#1}}	

\newcommand{\al}{\ensuremath{\alpha}}
\newcommand{\be}{\ensuremath{\beta}}

\newcommand{\de}{\ensuremath{\delta}}
\newcommand{\e}{\ensuremath{\varepsilon}}
\newcommand{\vp}{\ensuremath{\varphi}}
\newcommand{\la}{\ensuremath{\lambda}}

\newcommand{\ta}{\ensuremath{\theta}}

\newcommand{\R}{\ensuremath{\mathbb{R}}}

\newcommand{\Rd}{\ensuremath{{\mathbb{R}^{\dimension}}}}
\newcommand{\Rn}{\Rd}
\newcommand{\Z}{\ensuremath{\mathbb{Z}}}
\newcommand{\N}{\ensuremath{\mathbb{N}}}

\newcommand{\T}{\ensuremath{\mathbb{T}}}
\newcommand{\Tn}{{\T^\dimension}}

\DeclareMathOperator*{\halflimsup}{\star-limsup}
\DeclareMathOperator*{\halfliminf}{\star-liminf}

\newcommand{\Lip}{{\rm Lip}}

\definecolor{grey}{rgb}{0.6,0.6,0.6}

\numberwithin{equation}{section}

\newtheorem{theorem}{Theorem}[section]
\newtheorem{lemma}[theorem]{Lemma}
\newtheorem{proposition}[theorem]{Proposition}
\newtheorem{corollary}[theorem]{Corollary}

\newtheorem{definition}[theorem]{Definition}

\theoremstyle{definition}

\newtheoremstyle{remarkstyle}
  {3pt}
  {3pt}
  {\small}
  {}
  {\bfseries}
  {.}
  { }
  {}

\theoremstyle{remarkstyle}
\newtheorem{remark}[theorem]{Remark}
\newtheorem{example}[theorem]{Example}
\usepackage{paralist}

\DeclareMathOperator{\Wulff}{Wulff}

\DeclareMathOperator{\nbd}{\mathcal U}

\DeclareMathOperator{\pair}{Pair}
\DeclareMathOperator{\dom}{\mathcal D}
\newcommand{\parahead}[1]{\bigskip\noindent\textbf{{#1}.}}

\title[Anisotropic total variation flow]
{Anisotropic total variation flow of non-divergence type
on a higher dimensional torus}

\author[M.-H. Giga]{Mi-Ho Giga}
\author[Y. Giga]{Yoshikazu Giga}
\author[N. Po\v{z}\'{a}r]{Norbert Po\v{z}\'{a}r}

\date{}



\begin{document}

\begin{abstract}
We
extend the theory of viscosity solutions
to a class of very singular nonlinear parabolic problems
of non-divergence form
in a periodic domain of an arbitrary dimension
with diffusion given by an anisotropic
total variation energy.
We give a proof of a comparison principle,
an outline of a proof of the stability under
approximation by regularized parabolic problems,
and an existence theorem for general continuous initial data,
which extend the results recently obtained by the authors.

\end{abstract}

\maketitle

\gitversionfootnote


\section{Introduction}

The goal of this note is the announcement 
of the results in \cite{GGP13}
and their extension
to smooth anisotropic total variation energies.
Furthermore,
we give a slightly different
exposition of the technically demanding proof of
the comparison theorem.

In an arbitrary dimension $n\geq 1$
we consider the following problem
for a function $u(x,t): \Tn \times (0,T) \to \R$
on the torus $\Tn := \Rn \setminus \Z^n$
for some $T > 0$:
\begin{align}
\label{tvf}
    &u_t + F\pth{\nabla u,
    \divo \partial W\pth{\nabla u}} = 0
        & &\text{in $Q := \Tn \times (0,T)$,}
\intertext{with the initial condition}
\label{tvf-ic}
    &\at{u}{t=0} = u_0
        && \text{on $\Tn$.}
\end{align}
In this paper we assume that
\begin{align}
\label{W-regularity}
W \in C^2(\Rn \setminus \set0), \qquad \text{$W^2$ is strictly convex,}
\end{align}
and that $W$ is a convex one-homogeneous function,
positive outside of the origin, i.e.,
there exists a positive constant $\la_0$
such that
\begin{align}
\label{W-bound}
W(a p) = a W(p) \geq \la_0 a\abs{p}  \qquad \text{for all $p \in \Rn$, $a \geq 0$}.
\end{align}
Furthermore, we assume that
$F: \Rn \times\R \to \R$
is a continuous function,
non-increasing in the second variable, i.e.,
\begin{align}
\label{ellipticity}
F(p, \xi) &\leq F(p, \eta) &&\text{for $\xi,\eta\in \R$, $\xi \geq \eta$,
$p\in\Rn$.}
\end{align}
This makes the operator in \eqref{tvf}
\emph{degenerate parabolic}.

The symbol $\partial W$ denotes the subdifferential of $W$.
In general, the subdifferential of a convex lower semi-continuous function
$\vp$ on a Hilbert space $H$
endowed with a scalar product $\ang{\cdot, \cdot}_H$
is defined as the set
\begin{align*}
\partial \vp(x) := \set{v \in H : \vp(x + h) - \vp(x) \geq \ang{h, v}_H
\text{ for all $h \in H$}} \qquad x \in H.
\end{align*}
Since $W$ is not differentiable at the origin,
$\partial W(0)$
is not a singleton
and therefore an extra care has to be taken
when defining the meaning of the term
$\divo \partial W(\nabla u)$.
In fact,
we shall understand the term $\divo \partial W \pth{\nabla u}$
through
the subdifferential of the anisotropic total variation energy
on the Hilbert space $L^2(\Tn)$,
\begin{align}
\label{tv-energy}
E(\psi) :=
\begin{cases}
\int_\Tn W(\nabla \psi) & \psi \in L^2(\Tn) \cap BV(\Tn),\\
+\infty & \psi \in L^2(\Tn) \setminus BV(\Tn),
\end{cases}
\end{align}
where $BV(\Tn)$ is the space of functions of bounded variation
on $\Tn$.
We shall clarify this relation in Section~\ref{sec:nonlocal-curvature}.
Let us introduce the domain of the subdifferential $\partial E$
of the energy $E$ on $L^2(\Tn)$,
namely
\begin{align*}
\dom(\partial E) := \set{\psi \in L^2(\Tn):
\partial E(\psi) \neq \emptyset}.
\end{align*}
The problem \eqref{tvf}
can be written more rigorously as
\begin{align}
u_t + F(\nabla u, -\partial^0 E(u(\cdot, t))) = 0,
\end{align}
where $\partial^0 E$ is the minimal section (canonical restriction)
of the subdifferential $\partial E$ defined
for $\psi \in \dom(\partial E)$ as
\begin{align*}
\partial^0 E(\psi) \in \partial E(\psi)
\text{ such that }
\norm{\partial^0 E(\psi)}_{L^2(\Tn)} = \min_{v \in \partial E(\psi)} \norm{v}_{L^2(\Tn)}.
\end{align*}
Clearly $\partial^0 E(\psi)$ is well-defined and unique since
$\partial E(\psi)$ is a nonempty closed convex subset of $L^2(\Tn)$
whenever $\psi \in \dom(\partial E)$.

\parahead{Motivation}
The prototypical example
of \eqref{tvf}
is the total variation flow \cite{ACM04}
\begin{align}
\label{total-variation-flow}
u_t = \divo \pth{\frac{\nabla u}{\abs{\nabla u}}},
\end{align}
since $\partial W(\nabla u) = \set{{\frac{\nabla u}{\abs{\nabla u}}}}$
for $W(p) = \abs{p}$
when $\nabla u \neq 0$,
or more generally the anisotropic total variation
flow \cite{ACM02a,ACM02b,Moll05}.
This problem also explains the interpretation of $\divo \partial W(\nabla u)$
as the minimal section of $-\partial E(u)$.
Indeed, problem \eqref{total-variation-flow}
is formally the subdifferential inclusion
\begin{align*}
\begin{cases}
u_t \in - \partial E(u(t)) & t > 0,\\
u(0) = u_0 \in L^2(\Tn).
\end{cases}
\end{align*}
The theory of monotone operators due
to K\=omura \cite{Ko} and Br\'ezis \cite{Br} yields the existence of a unique
solution $u \in C([0,T], L^2(\Tn))$ that is moreover,
for all $t \in (0,T)$,
right-differentiable, $u(t) \in \dom(\partial E)$
and
\begin{align*}
\frac{d^+ u}{dt}(t) = - \partial^0 E(u(t))\qquad \text{for $t \in (0,T)$.}
\end{align*}

Nevertheless, our main motivation for the study of problem \eqref{tvf}
in its general non-divergence form
comes from the models of crystal growth.
Let us outline how problem \eqref{tvf}
can be heuristically derived as the graph formulation
of the motion of a surface by
the anisotropic crystalline curvature of a particular form.
Following the notation of \cite{BNP01a,BCCN06,B10},
we consider the surface energy functional
\begin{align}
\label{F-surface}
\mathcal{F}(\Gamma) := \int_\Gamma \phi^\circ(\nu) \diff{\mathcal H^{n}}
\end{align}
that measures the surface energy of the surface
$\Gamma = \partial K \subset \R^{n+1}$
of a body $K \subset \R^{n+1}$
with the unit outer normal vector $\nu$.
Here $\mathcal H^n$ is the $n$-dimensional Hausdorff measure
and $\phi^\circ$ is
a convex one-homogenous function
positive outside of the origin
given
as
\begin{align}
\label{phi-circ}
\phi^\circ(\eta) &= W(-p) + \abs{\eta_{n+1}}
&&\text{for all $\eta = (p, \eta_{n+1}) \in \R^{n+1}$}.
\end{align}
The Wulff shape of this surface energy is
the one-level set
\begin{align*}
\Wulff_{\phi} := \set{\eta \in \R^{n+1} :
    \phi(\eta) \leq 1}
\end{align*}
of the dual function
\begin{align*}
\phi(\xi) := \sup \set{\xi \cdot \eta: \eta \in \R^{n+1},
\ \phi^\circ(\eta) \leq 1}.
\end{align*}
Note that this makes $\phi^\circ$ the support function of $\Wulff_{\phi}$.
A simple computation shows that
\begin{align*}
\phi(\xi) = \max \set{W^\circ(-x), \abs{\xi_{n+1}}} \qquad
\xi = (x, \xi_{n+1}) \in \Rn\times\R,
\end{align*}
where
\begin{align*}
W^\circ(x) := \sup \set{x \cdot p: p\in\Rn,\ W(p) \leq 1}.
\end{align*}
Setting
\begin{align*}
\mathcal W := \set{x \in \Rn: W^\circ(x) \leq 1},
\end{align*}
we observe that
the Wulff shape of $\phi^\circ$
is a cylinder of length $2$ with base $-\mathcal W$,
that is,
\begin{align*}
\Wulff_{\phi} = (-\mathcal W) \times [-1, 1].
\end{align*}
The assumptions \eqref{W-regularity} and 
\eqref{W-bound} on $W$ guarantee that $W^\circ$
also satisfies \eqref{W-regularity} and \eqref{W-bound}
(possibly with a different $\la_0$).
In particular,
$\mathcal W$ is a strictly convex, compact set
with a $C^2$ boundary containing the origin
in its interior.

The first variation of the functional
$\mathcal F$
in \eqref{F-surface}
is called the crystalline mean curvature \cite{BNP01a,BNP01b}
\begin{align}
\label{mean-curvature}
\kappa_{\phi} := -\divo_{\phi,\tau} n_\phi^{\rm min},
\end{align}
where $\divo_{\phi, \tau}$ is the tangential divergence on $\Gamma$
with respect to $\phi$,
introduced in \cite{BNP01a},
and $n_\phi^{\rm min}$ is a so-called Cahn-Hoffman vector field
on $\Gamma$ that minimizes the norm of $\divo_{\phi,\tau} n_\phi$
in $L^2(\Gamma)$
with weight $\phi^\circ(\nu_\Gamma(\xi))$
among all Cahn-Hoffman vector fields $n_\phi$.
A Cahn-Hoffman vector field
is any vector field on $\Gamma$ that satisfies
$n_\phi(\xi) \in \partial \phi^\circ(\nu(\xi))$
where $\nu(\xi)$ is the unit outer normal vector of $K$
at $\xi$.
Since $\phi^\circ$ is not differentiable everywhere,
the vector field $n_\phi^{\rm min}$ might not be unique
but $\divo_{\phi, \tau} n_\phi^{\rm min}$ is unique \cite{BNP01a, GPR}.
We use a sign convention different from \cite{BNP01a} so that $\kappa_\phi$ equals to the conventional mean curvature in the direction of $\nu$
when $\phi(\xi) = |\xi|$.

Consider now a surface $\Gamma(t) \subset \R^{n+1}$, $t \geq 0$, that can be expressed as the graph
of a sufficiently smooth $\Z^n$-periodic function $u : \Tn\times \R \to \R$:
\begin{align*}
\Gamma(t) = \set{(x, u(x, t)) : x\in \Rn}\qquad t \geq 0,
\end{align*}
which is the boundary of the (crystal) body
$K(t) := \set{(x, \xi_{n+1}) : \xi_{n+1} < u(x,t)}$.
In the graph case, $\nu(\xi)$ for $\xi \in \Gamma(t)$
has the simple form \cite{Giga06}
\begin{align*}
\nu(x, u(x,t)) = \frac{(-\nabla u, 1)}{\sqrt{1 + \abs{\nabla u}^2}}.
\end{align*}
Using the definition of $\phi^\circ$ in \eqref{phi-circ},
we have the expression
\begin{align*}
\partial \phi^\circ(p, \eta_{n+1}) = \set{(x, 1) :
x\in -\partial W(-p)} \qquad \text{for $\eta_{n+1} > 0$.}
\end{align*} 
Therefore $n_\phi(\xi) = (-z_W(x), 1)$ for some vector field
$z_W(x) \in \partial W(\nabla u(x,t))$, $x \in \Rn$,
and the expression \eqref{mean-curvature} reduces 
for graphs $\Gamma(t)$ to
\begin{align*}
\kappa_\phi = \divo z_W^{\rm min}(x),
\end{align*}
where $\divo$ is the divergence on $\Tn$
and $z_W^{\rm min}$ minimizes the $L^2$-norm of $\divo z_W$
among all vector fields $z_W(x) \in \partial W(\nabla u(x))$ a.e.
such that $\divo z_W \in L^2(\Tn)$.
It turns out that $-\kappa_\phi$ coincides with the
minimal section of the total variation energy \eqref{tv-energy},
see Section~\ref{sec:subdifferential},
and therefore we shall formally write
\begin{align*}
\kappa_\phi = \divo \partial W(\nabla u)\quad
(= -\partial^0 E(u(\cdot,t))).
\end{align*}
The motion of $\Gamma(t)$ by the crystalline mean curvature
$\kappa_\phi$
can be written as
\begin{align}
\label{motion-by-mean-curv}
V = \kappa_\phi,
\end{align}
where $V$ is the normal velocity of $\Gamma(t)$
that can be expressed in terms of the derivatives of $u$ as
\cite{Giga06}
\begin{align*}
V = \frac{u_t}{\sqrt{1 + \abs{\nabla u}^2}}.
\end{align*}
Thus we can formally rewrite \eqref{motion-by-mean-curv}
for graphs as
\begin{align*}
u_t = \sqrt{1 + \abs{\nabla u}^2} \divo \partial W(\nabla u),
\end{align*}
which is not of divergence form,
but obviously can be cast in the form of \eqref{tvf}.

\parahead{Literature overview}
The motion by anisotropic crystalline mean curvature
has attracted significant attention
due to its importance in modeling of crystal growth.
The majority of articles follow one of the three
main approaches:
polygonal, variational and viscosity.

The \emph{polygonal approach} relies on the relatively simple
expression of the anisotropic crystalline curvature $\kappa_\phi$ for
curves in a two-dimensional plane.
In fact, the quantity $\kappa_\phi$
is constant on the flat line segments
that are parallel to the facets of
the Wulff shape $\Wulff_\phi$, and is inversely proportional to
the length of the line segment.
Therefore when the Wulff shape $\Wulff_\phi$
is a convex polygon, i.e.,
the energy density $\phi^\circ$ is ``crystalline'', it is possible to define
the evolution of polygonal
curves with sides parallel to the facets of $\Wulff_\phi$.
This special family of solutions, often referred to as
a crystalline flow or a crystalline motion, was introduced in
\cite{AG89,Taylor91}.
The validity of this approach is limited in higher dimensions
\cite{GGM}
because the quantity $\kappa_\phi$ might not be constant
or even continuous on the facets and facet-breaking and facet-bending
phenomena might occur \cite{BNP99,BNP01IFB}.
For a further development see also \cite{Ishiwata08}.

The \emph{variational approach} applies only to problems with a divergence structure.
One then understands $\kappa_\phi$
as a subdifferential of the corresponding singular interfacial energy.
It was shown in \cite{FG,EGS} that in such case
the crystalline motion can be interpreted
as the evolution given by the abstract theory
of monotone operators \cite{Br,Ko}.
In this approach, the crystalline motion can be approximated by
an evolution by smooth energies
and vice-versa,
or by a crystalline algorithm \cite{GirK,Gir}.

As we explained above, the curvature $\kappa_\phi$
might not be constant or even continuous on the facets
of bodies in dimension higher than two,
and facet breaking or bending might occur \cite{BNP99,BNP01IFB}.
In fact, $\kappa_\phi$ is in general only bounded and of bounded variation
on the facets \cite{BNP01a,BNP01b}
and a nontrivial obstacle problem has to be solved to calculate $\kappa_\phi$
\cite{BNP01IFB,GPR}.
The facets with constant curvature $\kappa_\phi$
are called \emph{calibrable} \cite{BNP01IFB}.
The convex calibrable sets were first characterized in two dimensions
by E. Giusti \cite{Giusti78} in the isotropic case $W(p) = \abs{p}$.
That result was extended recently to higher dimensions in \cite{ACC05},
and to anisotropic norms
in \cite{CCMN08}.
The concept of calibrable sets is related to
the so-called Cheeger sets \cite{KL06,CCN07,AC09}.

This suggests that the crystalline flow cannot be restricted
in dimensions higher than two
to bodies with facets parallel to the facets
of the Wulff shape $\Wulff_\phi$
and a more general class of solutions is necessary.
A notion of generalized solutions and a comparison principle
was established through an approximation by
reaction-diffusion equations in \cite{BN00,BGN00}
for $V_\nu = \phi \kappa_\phi$.
However, the existence is known only for convex compact initial data
\cite{BCCN06}.
Even in two dimensions, if there is a nonuniform driving force $c$
the abstract theory suggests that $\kappa_\phi + c$ might not be
constant on the facets \cite{GG98DS}.
This situation is important because $c$ is often non-constant
in the models of crystal growth.
However, if one allows to include bent polygons
with free boundaries corresponding to the endpoints of a facet,
it is possible to give a rather explicit solution
\cite{GR08,GR09,GGoR11}.
In the graph case in one-dimension,
there is also an approach that defines solutions via 
an original definition of composition of multivalued operators
that allows the study of the evolution of facets and the regularity of
solutions for a general class of initial data under
a non-uniform driving force $c$ \cite{MR08,KMR}.

\parahead{Viscosity solutions}
The third approach based on the theory of viscosity solutions  is the approach taken in this paper.
The merit is that one can prove existence and
uniqueness in a general class of continuous functions
without requiring a divergence structure of the problem,
only relying on the comparison principle.
The review paper \cite{GG04} compares
the viscosity and variational approaches
for equation of divergence form.

Since the operator in \eqref{tvf} has a parabolic structure,
it can be expected that any reasonable class of solutions of the problem
satisfies a comparison principle.
In particular, \eqref{tvf} should fall in the scope of
the theory of viscosity solutions.
Unfortunately, the conventional theory of degenerate parabolic
equations does not apply to \eqref{tvf}
because of the strong singularity of the operator
$\divo \partial W(\nabla \psi)$ on the facets of $\psi$,
that is, whenever $\nabla \psi = 0$.
Suppose that $\psi \in C^2(U)$
in an open set $U \subset \Rn$
and $\nabla \psi \neq 0$ in $U$.
Then $\partial W(\nabla \psi(x))$ is a singleton for $x \in U$
and $\divo \partial W(\nabla \psi)$ can be expressed as
\begin{align*}
\bra{\divo \partial W(\nabla\psi)}(x)=
k\pth{\nabla\psi(x), \nabla^2 \psi(x)},
\end{align*}
where
\begin{align*}
k(p,X):= \trace \bra{\nabla^2 W(p) X}\qquad p \in \Rn \setminus\set0, X \in \mathcal S^n.
\end{align*}
Here $\mathcal S^n$ is the set of symmetric $n\times n$-matrices.
Since $W$ is positively one-homogeneous,
$\nabla^2 \bra{W(ap)} = a^{-1} \nabla^2 W(p)$ for $a > 0$ and $p \in \Rn \setminus\set0$
and therefore
\begin{align}
\label{nonlocal-curv}
k(p,X) = \ov{\abs{p}} \trace \bra{\nabla^2 W\pth{\frac{p}{\abs{p}}} X}.
\end{align}
We observe that $k(p,X)$ is unbounded as $p \to 0$,
and, in fact, at $p = 0$ the diffusion is so strong that the operator
$\divo \partial W(\nabla \psi)$
becomes a nonlocal operator that depends on the shape
and size of the facet of $\psi$.
For this reason an equation with such operator
is often called a very singular diffusion equation \cite{FG,GG10}.
If the singularity of the operator $k(p,X)$
is relatively weak at $p = 0$ so that the operator is still local,
as in the case of the $q$-Laplace equation
$u_t - \divo (\abs{\nabla u}^{q-2} \nabla u) = 0$
for $1 < q < 2$,
which corresponds to $W(p) = \abs{p}^q/q$ in our notation,
the theory of viscosity solutions can be extended
\cite{Goto94,IS,OhnumaSato97,Giga06}.
Note that the level set formulation of the motion
by the mean curvature can also be written in the form of \eqref{tvf}
with $W(p) = \abs{p}$
and $F(p, \xi) = - \abs{p} \xi$.
However, the singularity of $k(p,X)$ in
\eqref{nonlocal-curv} is canceled out by $\abs{p}$ in $F(p,\xi)$
and the operator is bounded as $p \to 0$ \cite{CGG,ES}.
There has been a considerable effort to extend the theory
of viscosity solutions to the problem \eqref{tvf}
with a positively one-homogeneous $W$ and a general continuous $F$
satisfying only the monotonicity assumption \eqref{ellipticity}.
Until recently, however, the results have been restricted to
the one-dimensional case \cite{GG98ARMA,GG01ARMA,GGR11,GGNakayasu}
or to related level set equations for evolving planar curves
\cite{GG00Gakuto,GG01ARMA};
see also the review paper \cite{G04}.

\medskip
In the recent paper \cite{GGP13}, we extended the theory
of viscosity solutions to problem \eqref{tvf} with $W(p) = \abs{p}$.
In the present paper, we shall generalize this result
to an arbitrary $W$ that satisfies the assumptions above.

\parahead{Main results}
We introduce a notion of viscosity solutions
for problem \eqref{tvf}
and prove the following well-posedness result,
which is an extension of the main result in \cite{GGP13}.

\begin{theorem}[Main theorem]
Suppose that a continuous function $F : \Rn \times \R \to \R$
is degenerate elliptic in the sense of \eqref{ellipticity},
and that $W : \Rn \to \R$
satisfies \eqref{W-regularity} and \eqref{W-bound}.
Then the initial value problem \eqref{tvf}--\eqref{tvf-ic}
with $u_0 \in C(\Tn)$
has a unique global viscosity solution
$u \in C(\T^n \times [0,\infty))$.
If additionally $u_0 \in \Lip(\T^n)$,
i.e., $u_0$ is a periodic Lipschitz function,
then $u(\cdot, t) \in \Lip(\Tn)$ for all $t \geq 0$ and
\begin{align*}
    \norm{\nabla u(\cdot, t)}_\infty &\leq \norm{\nabla u_0}_\infty
    && \text{for $t \geq 0$.}
\end{align*}
\end{theorem}

As in \cite{GGP13},
the uniqueness of solutions
will be established via a comparison principle,
and the existence of solutions is verified
by showing the stability of solutions under approximation
by regularized problems.
As a corollary, we see that
in the case of the standard anisotropic total variation flow
equation our viscosity solutions
coincide with the semigroup (weak) solutions
given by the theory of monotone operators.

Viscosity solutions are defined as those functions that admit
a comparison principle with a class of test functions,
which are sufficiently regular functions to which the
operator in \eqref{tvf} can be applied directly.
The difficult task is the crafting of an appropriate
class of such test functions that is on one hand large enough
so that the viscosity solutions can be shown to be unique,
by the means of proving a comparison principle,
and on the other hand small enough so that the proof of existence
is possible for
any given sufficiently regular initial data.

As the computation above suggests,
we can evaluate the operator $\divo \partial W (\nabla \psi)$
at a point $x_0$
whenever $\psi \in C^2(U_{x_0})$ and $\nabla \psi(x_0) \neq 0$.
Thus arbitrary sufficiently smooth functions $\vp(x,t)$
with $\nabla \vp \neq 0$
serve as test functions.

However, the situation is much more delicate
at places where the gradient of the solution vanishes,
that is, on the facets.
The main difficulty stems from the restriction that
the operator $\divo \partial W(\nabla \psi) = - \partial^0 E(\psi)$
is only defined
for functions $\psi \in \dom(\partial E)$.
Fortunately,
a simple class of what we call \emph{faceted functions}
is available and we are able to show that such functions belong
to $\dom(\partial E)$ under some regularity assumptions
on the shape of the facet.
The main tool is the characterization of the subdifferential $\partial E(\psi)$ of a Lipschitz function $\psi$ (Corollary~\ref{co:char-lip}).
Namely, a function belongs to $\partial E(\psi)$
if it is the distributive divergence
of a vector field that pointwise almost everywhere belongs
to the sets $\partial W(\nabla \psi(x))$.
To construct a Lipschitz faceted function,
we start from a pair of sets
that satisfy certain regularity conditions
and characterize the facet.
This characterization follows from the simple observation that any facet 
of a continuous function
$\psi$
can be uniquely described by a pair of disjoint open sets
$\set{\psi > a}$ and $\set{\psi < a}$ for some $a \in \R$.
The quantity $- \partial^0 E$ is well-defined for such faceted functions,
and, moreover,
if two pairs are ordered in a specific sense,
the values of $-\partial^0 E$ are also ordered on the intersection
of the facets.

In contrast to \cite{GGP13},
we do not introduce the quantity $\Lambda$,
which we refer to as the \emph{nonlocal curvature}
of a facet there. This makes the current exposition more straightforward.

The definition of viscosity solutions (Definition~\ref{def:visc-sol})
then contains the classical test with smooth test functions when the gradient
of the solution is nonzero,
and a new faceted test with a class of faceted test functions.
In the faceted test we only evaluate the essential infima and suprema of
$-\partial^0 E$ over balls of small radius and thus obtain a pointwise
quantity.
Furthermore, to facilitate the proof of stability and existence,
we require that the faceted test function can be shifted in an arbitrary
direction by a small amount,
that is, we say that the faceted test function is in \emph{general position}.

The proof of the comparison principle (Theorem~\ref{th:comparison})
follows the standard doubling-of-variables argument
with an important twist.
Suppose that $u$ and $v$ are viscosity solutions of \eqref{tvf}
such that $u(\cdot, 0) \leq v(\cdot, 0)$.
We introduce an extra parameter $\zeta \in \Tn$
and investigate the $\zeta$-dependence of the maxima of the functions
\begin{align*}
\Phi_\zeta(x,t,y,s;\e) := u(x,t) - v(y,s) - \frac{\abs{x - y - \zeta}^2}{2\e}
-S(t,s; \e)
\end{align*}
over $(x,t,y,s) \in \Tn \times [0,T] \times \Tn \times [0,T]$
and a fixed parameter $\e > 0$. The time penalization $S(t,s;\e)$
is defined in Section~\ref{sec:comparison-principle}.
This device was developed in \cite{GG98ARMA}, but its history goes back to
\cite{CGG,Goto94}.
In particular, by varying $\zeta$, we increase the change that
some maximum will occur at a point $(x,t,y,s)$ such that $x - y - \zeta \neq 0$ and the standard construction of a test function for
the classical test with nonzero gradient is available \cite{CIL,Giga06}.
If all maxima of $\Phi_\zeta$ for 
all small $\zeta$ happen to lie at points $(x,t,y,s)$ such that
$x - y - \zeta = 0$,
we get extra information about the shape of $u$ and $v$
at their contact point.
To be more specific, $u$ and $v$ must have some flatness
and therefore there is enough room for finding two ordered smooth pairs that
can be used to construct ordered faceted test functions
for both $u$ and $v$.

The existence of solutions (Theorem~\ref{th:existence})
follows from the stability
under approximation by regularized degenerate parabolic problems
(Theorem~\ref{th:stability})
for which the standard theory of viscosity solutions applies
\cite{CIL}.
We regularize \eqref{tvf}
through an approximation of $W$ by a decreasing sequence of strongly convex
smooth functions $W_m$, $m \geq 1$, with a quadratic growth at infinity,
so that the subdifferential $-\partial^0 E_m$ of the corresponding energy $E_m(\psi) := \int W_m(\nabla u)$
is a uniformly elliptic quasi-linear differential operator.

Since we approximate a nonlocal problem by local problems,
the main difficulty materializes
while passing through the limit in the definition of viscosity solutions.
More precisely,
when we apply the regularized operator to a (smooth) faceted test function,
we recover only local information that is independent
of the overall shape of the facet,
while in the limit the shape of the facet is very important.

To recover the nonlocal information,
we perturb the test function $\vp(x,t) = \psi(x) + g(t)$
by one step of the implicit Euler
approximation of the anisotropic total variation flow
with time-step $a>0$,
that is, by finding the solution $\psi_a$ of the resolvent problem
\begin{align*}
\psi_a = (I + a \partial E)^{-1} \psi.
\end{align*}
By solving the resolvent problem for the regularized energy $E_m$,
\begin{align*}
\psi_{a,m} = (I + a \partial E_m)^{-1} \psi,
\end{align*}
we obtain a smooth perturbed test function
$\vp_{a,m}(x,t) = \psi_{a,m}(x) + g(t)$
for the regularized problem that contains the
missing nonlocal information.
This type of approximation has two advantages.
Firstly,  $\psi_a$ is uniformly approximated by
$\psi_{a,m}$ as $m\to\infty$ for a fixed $a$
and so is $\psi$ by $\psi_a$ as $a\to0$.
Secondly,
if $\psi \in \dom(\partial E)$
then the function $-\partial^0 E(\psi)$
is approximated in $L^2(\Tn)$
as $a \to 0+$
by the ratio $(\psi_a-\psi)/a$.
This is the main ingredient in the proof of stability.

To finish the proof of existence, we have to show that the limit
of solutions of the regularized problem has the correct initial data.
This is done by a comparison with barriers at $t = 0$.
However, it is necessary to construct barriers depending on $m$.
As in \cite{GG98ARMA} and \cite{GGP13}, we use the
convex conjugates of $W_m$, but with a more robust
cutoff of large gradients that requieres neither one-dimensionality
nor radial symmetry of $W_m$.

\parahead{Outline}
This paper consists of the following parts.
First, in Section~\ref{sec:nonlocal-curvature},
we discuss the interpretation
of the term $\divo \partial W(\nabla \psi) \sim -\partial^0 E(\psi)$
for a class of functions $\psi$ that have flat parts,
the so-called facets.
This will be then used in Section~\ref{sec:viscosity-solutions}
to introduce viscosity solutions of problem \eqref{tvf}
and a suitable class of test functions.
Once the solutions are defined,
we establish a comparison principle in Section~\ref{sec:comparison-principle}.
The paper is concluded in Section~\ref{sec:existence-stability}
with a brief discussion of
stability of \eqref{tvf} under approximation
by regularized problems,
which provides, as a corollary, the existence of solutions.

\section{Nonlocal curvature}
\label{sec:nonlocal-curvature}

The main challenge for developing
a reasonable theory of viscosity solutions
is the selection of an appropriate class of test functions.
In particular,
a special care has to be taken
when the gradient of a solution vanishes.
In such a case,
the solution should have a facet,
i.e., it should be constant on
a closed neighborhood of the point.
Functions that have such facets will be called faceted functions.

In this section we will investigate the value
of the term $\divo \partial W(\nabla \psi)
\sim -\partial^0 E(\psi)$
on facets of faceted functions.
It turns out that such facets can be described
by a pair of disjoint open sets, which
characterize the convexity and concavity
of the functions at the facet boundary.
The understanding of the term $-\partial^0 E(\psi)$
is further complicated by the fact that
it is a nonlocal quantity on facets.
Motivated by the motion by crystalline mean curvature,
we shall refer to this term as the \emph{nonlocal curvature},
in particular if this term is evaluated
on a facet.
Instead of evaluating it directly,
we approximate it via a resolvent problem
for the energy $E$.
This both yields a comparison principle for
$-\partial^0 E(\psi)$
and a way how the approximate it
via regularized energies
in the proof of existence.

In contrast to \cite{GGP13},
we do not define the quantity $\Lambda$
which we called nonlocal curvature there
and showed that it is independent of the choice of support function
of a given pair.
The proof of this fact is quite technical,
but it is extendable to the current context.
However, this quantity is not necessary for definition
of viscosity solutions
and we choose a more direct approach here.

\subsection{Torus}

We consider the total variation energy for periodic functions
on $\Rn$.
These functions can be identified with functions
on the $n$-dimensional torus $\Tn := \Rn / \Z^n$.
The set $\Tn$ is the set of all equivalency classes
$\set{x + \Z^n : x \in \Rn}$
with the induced metric and topology, namely
\begin{align}
\label{torus-metric}
\dist(x,y) := \dist_\Rn(x + \Z^n, y + \Z^n),\quad
\abs{x} := \dist(x, 0) = \inf_{k \in \Z^n} \abs{x + k}_\Rn,
\end{align}
for $x,y \in \Tn$.
Consequently, the open ball $B_r(x)$
centered at $x \in \Tn$ of radius $r > 0$
is defined as
$B_r(x) := \set{y \in \Tn: \abs{x-y} < r}$.
Note that $B_r(x)$ has a smooth boundary if $r < 1/2$.

\subsection{Subdifferential of the total variation energy}
\label{sec:subdifferential}

Function $u$ is called a function of bounded variation
and said to belong to $BV(\Tn)$ if $u \in L^1(\Tn)$
and its gradient $Du$ in the sense of distributions
is a vector valued Radon measure
with finite total variation on $\Tn$.

To characterize the subdifferential of $E$,
we need a pairing between functions of bounded
variations
and vector fields with $L^2$ divergence
that was studied in \cite{Anzellotti} (see also \cite{FM})
for bounded domains in $\Rn$.
The modification for $\Tn$ is straightforward.
We recall the definition of the space of vector fields
\begin{align*}
X_2(\Tn) := \set{z \in L^\infty(\Tn; \Rn): \divo z \in L^2(\Tn)}.
\end{align*}
It was also shown in \cite{Anzellotti} that
for any $z \in X_2(\Tn)$ and $u \in BV(\Tn) \cap L^2(\Tn)$
we can define a Radon measure $(z, Du)$ on $\Tn$ as
\begin{align*}
\ang{(z,Du), \vp} :=
-\int_\Tn u \vp \divo z - \int_\Tn u z \cdot \nabla \vp
\qquad \vp \in C^\infty(\Tn).
\end{align*}

The following characterization of the subdifferential
of energy $E$ was proved in \cite{Moll05}
on subsets of $\Rn$,
but a modification for $\Tn$ is straightforward.

\begin{proposition}
Let $u, v \in L^2(\Tn)$.
Then $v \in \partial E(u)$
if and only if
$u \in BV(\Tn)$
and there exists a vector field $z \in X_2(\Tn)$
such that $z(x) \in \partial W(\nabla u(x))$ a.e.,
$(z, Du) = W(Du)$ as measures in $\Tn$
and $v = - \divo z$.
\end{proposition}

\begin{remark}
Since $W$ is one-homogeneous,
we can define the measure $W(Du) := 
W(\nabla u) + W\pth{\frac{D^s u}{\abs{D^s u}}} \abs{D^s u}$
for any $u \in BV(\Tn)$,
where $\nabla u$ is the absolutely continuous
part of $Du$ with respect to the Lebesgue measure
and $D^s u$ is the singular part.
\end{remark}

However, for our purposes we only need the characterization
of the subdifferential for Lipschitz test functions,
in which case we get the following simpler corollary.

\begin{corollary}
\label{co:char-lip}
Let $u \in \Lip(\Tn)$ and $v \in L^2(\Tn)$.
Then $v \in \partial E(u)$
if and only if
there exists a vector field $z \in X_2(\Tn)$
such that $z(x) \in \partial W(\nabla u(x))$ a.e.
and $v = - \divo z$.
\end{corollary}

\begin{remark}
\label{re:subdiff-cone}
It is clear from Corollary~\ref{co:char-lip} that if $\psi \in \Lip(\Tn)$
and $v \in \partial E(\psi)$
then for any positive constants $\al, \be > 0$
we have
$v \in \partial E(\hat\psi)$
where
\begin{align*}
\hat \psi = \al [\psi]_+ - \be [\psi]_-
\end{align*}
where $[s]_\pm := \max (\pm s, 0)$;
see \cite[Remark~3.2]{CasellesChambolle06}.
In particular, $\partial E(\psi) = \partial E(\hat \psi)$.
This is a consequence of the one-homogeneity and convexity of $W$
which imply that $\partial W(p) = \partial W(a p)$
and $\partial W(p) \subset \partial W(0)$
for all $p \in \Rn$, $a > 0$.
\end{remark}

\subsection{General facets}

By $\mathcal P$ we shall denote
all ordered pairs of disjoint subsets of $\Tn$.
Additionally, $(\mathcal P,\preceq)$ will be a partially ordered set with
ordering
\begin{align*}
(A_-, A_+) \preceq (B_-, B_+)
\qquad \Leftrightarrow \qquad
A_+ \subset B_+ \text{ and } B_- \subset A_-
\end{align*}
for $(A_-, A_+), (B_-, B_+) \in \mathcal P$.
We will also denote the reversal by
\begin{align*}
-(A_-, A_+) := (A_+, A_-).
\end{align*}
By definition, if $(A_-, A_+) \preceq (B_-, B_+)$
then $-(B_-, B_+) \preceq -(A_-, A_+)$.

\begin{definition}
A pair $(A_-, A_+)\in \mathcal P$ is \emph{open}
if both sets $A_\pm$ are open.

We say that $\psi \in \Lip(\Tn)$ is a \emph{support function}
of an open pair $(A_-, A_+) \in \mathcal P$
if
\begin{align*}
\psi
\begin{cases}
> 0 & \text{in } A_+,\\
= 0 & \text{in } (A_- \cup A_+)^c,\\
< 0 & \text{in } A_-.
\end{cases}
\end{align*}
On the other hand,
for any function $\psi$ on $\Tn$
we define its pair (not necessarily open)
\begin{align*}
\pair(\psi) := (\set{x : \psi(x) < 0}, \set{x : \psi(x) > 0}).
\end{align*}
\end{definition}

\begin{remark}
\label{re:support-function-symmetry}
If $\psi$ is a support function of an open pair $(A_-, A_+) \in \mathcal P$
then $-\psi$ is a support function of the open pair
$-(A_-, A_+) := (A_+, A_-)$.
With this notation we have
\begin{align*}
\pair(\psi) = -\pair(-\psi)
\end{align*}
for any function $\psi$.
\end{remark}

\begin{example}
For any open pair $(A_-, A_+) \in \mathcal P$
the function
\begin{align*}
\psi(x) := \dist(x,A_+^c) - \dist(x, A_-^c)
\end{align*}
is a support function of $(A_-, A_+)$.
\end{example}

\begin{definition}
We say that an open pair $(A_-, A_+) \in \mathcal P$
is a smooth pair
if
\begin{enumerate}[(i)]
\item
$\dist(A_-, A_+) > 0$,
\item
$\partial A_- \in C^\infty$ and $\partial A_+ \in C^\infty$.
\end{enumerate}
Note that this definition allows for $A_-$ and/or $A_+$
to be empty as $\dist$ is $+\infty$ by definition
when one of the sets is empty. 
\end{definition}

\begin{definition}
We say that an open pair $(A_-, A_+) \in \mathcal P$
is an admissible pair
if
there exists a support function $\psi$ of $(A_-, A_+)$
such that $\psi \in \dom(\partial E)$.
\end{definition}

We shall show that every pair in $\mathcal P$
can be approximated in Hausdorff distance
by a smooth pair,
and in turn that every smooth pair is an admissible pair.

The main tool in the construction will be
the generalized $\rho$-neighborhood
of a set $A$, defined as
\begin{align*}
\nbd^\rho(A) :=
\begin{cases}
A + \cl B_\rho(0) & \rho > 0,\\
A & \rho = 0,\\
\set{x \in \Tn: \cl B_\rho(x) \subset A} & \rho < 0,
\end{cases}
\end{align*}
where $G+H := \set{x+y: x \in G,\ y \in H}$
denotes the Minkowski sum of sets and
$\cl B_\rho(x)$ is the closed ball of radius $\rho$ centered
at $x$.
In image analysis it is often written as
$\nbd^\rho(A) = A \oplus \cl B_\rho(0)$ for $\rho > 0$
and $\nbd^{\rho}(A) = A \ominus \cl B_{\abs\rho}(0)$ for $\rho < 0$,
where $\oplus$ denotes the Minkowski addition and $\ominus$
denotes the Minkowski decomposition.
In morphology, $\oplus$ is called dilation
and $\ominus$ is called erosion.
We collect the basic properties of $\nbd^\rho$ in the following
proposition; its proof is quite straightforward.

\begin{proposition}
\label{pr:nbd-properties}
\begin{enumerate}
\item $\mathcal U^{-\rho}(A) \subset A \subset \mathcal U^\rho(A)$
for $\rho > 0$.
\item (complement)
\begin{align}
\label{compl-nbd}
\pth{\nbd^\rho(A)}^c = \nbd^{-\rho}(A^c)
\qquad \text{for any set $A \subset \Tn$ and $\rho \in \R$}
\end{align}
\item (monotonicity)
\begin{align*}
\nbd^\rho(A_1) \subset \nbd^\rho(A_2)\qquad
\text{for $A_1 \subset A_2 \subset \Tn$ and $\rho \in \R$.}
\end{align*}
\item
$\nbd^\rho(A_1 \cap A_2) \subset \nbd^\rho(A_1) \cap \nbd^\rho(A_2)$
 for all $\rho \in \R$, with equality for $\rho \leq 0$.
\item
$\nbd^r(\nbd^\rho(A)) \subset \nbd^{r+\rho}(A)$
for $r \geq 0$ and $\rho \in \R$; equality holds if $\rho \geq 0$.
\item
For any $\rho \in \R$, we have
$\nbd^\rho(A_1) \subset A_2$ if and only if $A_1 \subset \nbd^{-\rho} (A_2)$.
\end{enumerate}
\end{proposition}

For a set $A \subset \Tn$ we introduce the signed distance function
\begin{align*}
d_A(x) := \dist(x, A) - \dist(x, A^c).
\end{align*}
We observe that
\begin{align*}
\interior \nbd^\rho(A) &= \set{x \in \Tn : d_A(x) < \rho},\\
\nbd^\rho(\cl A) &= \set{x \in \Tn: d_A(x) \leq \rho}
\end{align*}
for all $\rho \in \R$.

For pair $(A_-, A_+) \in \mathcal P$ we define the $\rho$-neighborhood as
\begin{align*}
\mathcal U^\rho(A_-, A_+) :=
(\mathcal U^{-\rho}(A_-), \mathcal U^{\rho}(A_+)).
\end{align*}
Clearly
\begin{align*}
\mathcal U^{-\rho}(A_-, A_+)
\preceq (A_-, A_+) \preceq \mathcal U^\rho(A_-, A_+) \qquad \rho \geq 0.
\end{align*}

The following lemma was proved in \cite{GGP13}.

\begin{lemma}
\label{le:smooth-nbd}
For any set $A \subset \Rn$
and constants $\rho_1, \rho_2$, $0 < \rho_1 < \rho_2$,
there exist open sets $G_-, G_+ \subset \Rn$
with smooth boundaries
such that
\begin{align*}
\mathcal U^{-\rho_2}(A) \subset G_- \subset \mathcal U^{-\rho_1}(A)
\subset
A \subset \mathcal U^{\rho_1}(A) \subset G_+ \subset \mathcal U^{\rho_2}(A).
\end{align*}
\end{lemma}

Using the previous lemma,
we can show that any pair in $\mathcal P$ can be approximated
in Hausdorff distance by a smooth pair.

\begin{proposition}
\label{pr:smooth-pair-approx}
Let $(A_-, A_+) \in \mathcal P$ be a pair
and let $0 \leq \rho_1 < \rho_2$.
Then there exists a smooth pair $(G_-, G_+) \in \mathcal P$
such that
\begin{align}
\label{smooth-pair-approx}
\nbd^{\rho_1}(A_-, A_+) \preceq (G_-, G_+) \preceq \nbd^{\rho_2}(A_-, A_+).
\end{align}
\end{proposition}

\begin{proof}
Let us set $\de := (\rho_2 - \rho_1)/3 > 0$.
We apply Lemma~\ref{le:smooth-nbd}
to the set $A_+$
and obtain a smooth set $G_+$ such that
\begin{align*}
\nbd^{\rho_1}(A_+) \subset G_+ \subset \nbd^{\rho_1 + \de}(A_+).
\end{align*}
Then we apply Lemma~\ref{le:smooth-nbd}
to the set $A_-$ 
and obtain a smooth set
and $G_-$
such that
\begin{align*}
\nbd^{-\rho_2}(A_-) \subset G_- \subset \nbd^{-\rho_2 + \de} (A_-).
\end{align*}
We claim that $\dist(G_-, G_+) \geq \de$.
Indeed, we can assume that both $G_-$ and $G_+$ are nonempty
and we choose any $x \in G_+$, $y \in G_-$
and $z \in A_+$.
Since by definition of $G_-$
we have $\dist (y, A_-^c) \geq \rho_2 - \de$
and $z \in A_+ \subset A_-^c$, clearly $\dist(y, z) \geq \rho_2 - \de$.
Therefore
\begin{align*}
\rho_1 + 2\de = \rho_2 - \de \leq \dist(y, z) \leq
\dist(y, x) + \dist(x, z).
\end{align*}
Since $\inf_{z\in A_+} \dist(x,z) = \dist(x, A_+) \leq \rho_1 + \de$
by the definition of $G_+$, we conclude that
$\dist (G_-, G_+) =
\inf_{y \in G_-} \inf_{x \in G_+} \dist (x,y) \geq \de$.

Therefore $(G_-, G_+)$ is a smooth pair
and by construction
\eqref{smooth-pair-approx} holds.
\end{proof}

Finally,
every smooth pair is an admissible pair.

\begin{proposition}
Suppose that $(G_-, G_+)\in \mathcal P$
is a smooth pair.
Then there exists a support function $\psi$ of $(G_-, G_+)$
such that $\psi \in \dom(\partial E)$.
\end{proposition}

\begin{proof}
Since $\partial G_\pm$ is smooth and $\Tn$ is compact,
there exists $\de_\pm$ such that
$d_{G_\pm}$ is smooth in the set
$\set{x : d_{G_\pm} < \de_\pm}$; see \cite{DZ11}.
Let us take
\begin{align*}
\de := \ov3 \min \set{\de_-, \de_+, \dist(G_-, G_+)} > 0.
\end{align*}
Introduce the cutoff functions $\chi \in \Lip(\R)$
and $\ta \in C^\infty_c(\R)$
such
that
\begin{align*}
\chi(s) := \max(0, \min(\de, s))
\end{align*}
and $\ta(s) \in [0,1]$ with $\ta(s) = 1$ on $[0,\de]$
and $\ta(s) = 0$ on $\R \setminus (-\de, 2\de)$.

We define
\begin{align*}
\psi(x) := \chi(d_{G_+^c}(x)) - \chi(d_{G_-^c}(x))
=\min \set{\de, \dist(x, G_+^c)} - \min \set{\de, \dist(x, G_-^c)}
\end{align*}
and a vector field
\begin{align*}
z(x) = 
\ta(d_{G_+^c}(x)) \partial^0 W(\nabla d_{G_+^c}(x))
+\ta(d_{G_-^c}(x)) \partial^0 W(-\nabla d_{G_-^c}(x)).
\end{align*}

Clearly $\psi \in \Lip(\Tn)$,
$z \in \Lip(\Tn)$
and $\psi$ is a support function of $(G_-, G_+)$.
It is also easy to see that
$z(x) \in \partial W(\nabla \phi(x))$ for a.e. $x \in \Tn$.
In particular, $-\divo z \in \partial E(\psi)$
and therefore $\psi \in \dom(\partial E)$
by Corollary~\ref{co:char-lip}.
\end{proof}

\subsection{Resolvent equation}

It is possible to approximate the minimal section of the subdifferential
$-\partial^0 E$
via a resolvent problem on $\Tn$.
That is, for given $\psi \in L^2(\Tn)$ and $a > 0$
find $\psi_a \in L^2(\Tn)$ that satisfies
\begin{align}
\label{resolvent-problem}
\psi_a + a \partial E(\psi_a) \ni \psi.
\end{align}
The standard theory of calculus of variations
yields that this problem has a unique solution
$\psi_a \in \dom(\partial E)$;
see \cite{Evans}.
We have the following well-known result \cite{Attouch, Evans}.

\begin{proposition}
\label{pr:resolvent-convergence}
If $\psi \in \dom(\partial E)$
then
\begin{align*}
\frac{\psi_a - \psi}{a} \to -\partial^0 E(\psi)
\qquad \text{in $L^2(\Tn)$ as $a \to 0$,}
\end{align*}
where $\psi_a$ is the unique solution of \eqref{resolvent-problem}
\end{proposition}

Moreover, a comparison theorem for  \eqref{resolvent-problem}
was proved in \cite{CasellesChambolle06}.

\begin{proposition}
\label{pr:resolvent-comparison}
Let $\psi^1_a$, $\psi^2_a \in L^2(\Tn)$
be two solutions of \eqref{resolvent-problem}
with $a > 0$ and
right-hand sides $\psi^1, \psi^2 \in L^\infty(\Tn)$,
respectively.
If $\psi^1 \leq \psi^2$ then $\psi_a^1 \leq \psi_a^2$.
\end{proposition}

\subsection{Monotonicity of nonlocal curvatures}

First, we state a useful lemma for generating
support functions in the domain of the subdifferential
$\partial E$
given an admissible pair and an upper semi-continuous function.

\begin{lemma}
\label{le:support-construction}
Let $\ta \in USC(\Tn)$
and let $(G_-, G_+) := \pair(\ta)$.
Suppose that $(H_-, H_+) \in \mathcal P$
is an admissible pair
and that there exists $\delta > 0$
such that
\begin{align*}
(G_-, G_+) \preceq \nbd^{-\de}(H_-, H_+).
\end{align*}
Then there exists a support function
$\psi$ of $(H_-, H_+)$
such that $\psi \in \dom(\partial E)$
and
\begin{align*}
\ta \leq \psi \qquad \text{on $\Tn$.}
\end{align*}
If, moreover, $\hat \psi \in \dom(\partial E)$
is a support function of $(H_-, H_+)$,
we can take $\psi$ such that
$-\partial^0 E(\psi) = -\partial^0 E(\hat \psi)$.
\end{lemma}

\begin{proof}
Since $(H_-, H_+)$
is an admissible facet, there exists
a support function $\psi_H \in \dom(\partial E)$.
By the definition of $(G_-, G_+)$ and $\psi_H$,
we immediately have that $\ta \leq \psi_H$
on $G_+^c \cap H_-^c$.
We will modify the function $\psi_H$ on the rest of $\Tn$
to guarantee that the ordering holds on the whole $\Tn$.
From the strict ordering of the pairs by $\de > 0$,
we immediately get
\begin{align*}
\cl{G_+} \subset H_+, \qquad \cl{H_-} \subset G_-.
\end{align*}
We define a new support function of $(H_-, H_+)$
as
\begin{align*}
\psi(x) := \al [\psi_H]_+ - \be [\psi_H]_-,
\end{align*}
where $\al$ and $\be$ are given positive constants specified below
and $[\cdot]_+$ and $[\cdot]_-$ are the positive
and negative parts.
$\psi$ is still a support function of $(H_-, H_+)$
and Remark~\ref{re:subdiff-cone} yields that
$\psi \in \dom(E)$.

We shall determine the constants $\al$ and $\beta$.
If $G_+ = \emptyset$
then $\ta \leq \psi_H$ on $G_+$ trivially
and we set $\al = 1$.
Otherwise, by compactness, semi-continuity
and the definition of support functions,
we have
\begin{align*}
\al := \frac{\max_\Tn \ta}{\min_{\cl{G_+}} \psi_H} > 0.
\end{align*}
Similarly,
if $H_- = \emptyset$ we set $\be = 1$,
otherwise
\begin{align*}
\qquad \be := \frac{\max_{\cl{H_-}} \ta}{\min_\Tn \psi_H} > 0.
\end{align*}
We observe that such a choice of $\al$ and $\be$
guarantees that
\begin{align}
\label{psi-order}
\ta \leq \psi \qquad \text{on } \Tn.
\end{align}

Finally,
we can take $\psi_H = \hat \psi$.
Then Remark~\ref{re:subdiff-cone} yields that
$-\partial^0 E(\psi) = -\partial^0 E(\psi_H)$.
\end{proof}

The following monotonicity result plays the role of a comparison principle for
admissible pairs.
The analogous result in \cite{GGP13} was stated for ordered smooth pairs,
and thanks to this extra regularity we did not need to
assume that the pairs are ordered strictly.

\begin{proposition}
\label{pr:monotonicity}
Suppose that $(G_-, G_+) \in \mathcal P$
and $(H_-, H_+) \in \mathcal P$
are two open pairs
that are moreover strictly ordered, i.e.,
there exists $\de > 0$ such that
\begin{align*}
\nbd^\de(G_-, G_+) \preceq (H_-, H_+).
\end{align*}
Then
for any support function $\psi_G$ of $(G_-, G_+)$
and any support function $\psi_H$ of $(H_-, H_+)$
such that $\psi_G, \psi_H \in \dom(\partial E)$
we have
\begin{align*}
-\partial^0 E(\psi_G) \leq
-\partial^0 E(\psi_H)
\qquad
\text{a.e. on $G_-^c \cap G_+^c \cap H_-^c \cap H_+^c$.}
\end{align*}
\end{proposition}

\begin{proof}
We apply the comparison principle for the resolvent problem
\eqref{resolvent-problem};
it is also possible to use the evolution equation
as in \cite{GGM}.

Let us denote the intersection of the facets as $D$,
\begin{align*}
D := G_-^c \cap G_+^c \cap H_-^c \cap H_+^c.
\end{align*}
We can assume that $\psi_G \leq \psi_H$.
Indeed, if this ordering does not hold
we replace $\psi_H$ with the function $\psi$
provided by 
Lemma~\ref{le:support-construction}
applied with $\ta = \psi_G$ and $\hat \psi = \psi_H$
since
$-\partial^0 E(\psi) = - \partial^0 E(\psi_H)$.

Clearly, the support functions coincide with zero
on the intersection of the facets, i.e., 
\begin{align}
\label{psi-zero}
\psi_G = \psi_H = 0 \qquad \text{on $D$.}
\end{align}

For each $a > 0$,
we find the solution $\psi^i_a$ of
the resolvent problem \eqref{resolvent-problem}
with right-hand side $\psi_i$, $i=G,H$.
Due to the $L^2$ convergence in Proposition~\ref{pr:resolvent-convergence},
we can find a subsequence $a_k \to 0$ as $k \to \infty$
such that
$(\psi^i_{a_k} - \psi_i)/a_k \to -\partial^0 E(\psi_i)$
 a.e. on $\Tn$ as $k \to \infty$
for $i=G,H$.

The comparison principle, Theorem~\ref{pr:resolvent-comparison},
and \eqref{psi-order}
imply that $\psi^G_{a_k} \leq \psi^H_{a_k}$.
Moreover, by \eqref{psi-zero},
$\psi^i_{a_k} - \psi^i = \psi^i_{a_k}$
on $D$
for all $k$.
Therefore
\begin{align*}
-\partial^0 E(\psi_G)
    &= \lim_{k\to\infty} \frac{\psi^G_{a_k}}{a_k}\\
    &\leq \lim_{k\to\infty} \frac{\psi^H_{a_k}}{a_k}
    = -\partial^0 E(\psi_H)
    &&\text{a.e. in $D$}
\end{align*}
and the comparison principle for $-\partial^0 E$ is established.
\qedhere\end{proof}

\section{Viscosity solutions}
\label{sec:viscosity-solutions}

This section finally introduces viscosity solutions of \eqref{tvf}.
As in the previous work \cite{GGP13},
it is necessary to separately define
test functions for 
the zero gradient of a solution and the nonzero gradient.
In this section we work on the parabolic cylinder
$Q := \Tn \times (0,T)$ for some
$T > 0$.

\begin{definition}
Let $(A_-, A_+) \in \mathcal P$ be a smooth pair
and let $\hat x \in \Tn \setminus \cl{A_- \cup A_+}$.
Function $\varphi(x,t) = \psi(x) + g(t)$,
where $\psi \in \Lip(\Tn)$ and $g \in C^1(\R)$, is
called an \emph{admissible faceted test function}
at $\hat x$
with a pair $(A_-, A_+)$
if $\psi \in \dom(\partial E)$
and $\psi$ is a support function
of the pair $(A_-, A_+)$.
\end{definition}

\begin{definition}
We say that an admissible faceted function $\varphi$
at $\hat x$ with a pair $(A_-, A_+)$
is in a general position of radius $\eta > 0$
with respect to $u : \cl Q \to \R$ at
$(\hat x, \hat t) \in Q$ if
$\cl B_\eta(\hat x) \subset \Tn \setminus \cl{A_- \cup A_+}$
and
\begin{align*}
u(x, t) - \inf_{h \in \cl B_\eta(0)} \vp (x - h, t)
\leq u(\hat x, \hat t) - \vp(\hat x, \hat t)
\qquad \text{for all $x \in \Tn$, $t \in [\hat t - \eta, \hat t + \eta]$}.
\end{align*} 
\end{definition}

\begin{definition}[Viscosity solutions]
\label{def:visc-sol}
An upper semi-continuous function $u : \cl Q \to \R$
is a \emph{viscosity subsolution}
of \eqref{tvf}
if
the following holds:

\begin{enumerate}[(i)]

\item (\emph{faceted test})
If $\vp(x,t) = \psi(x) + g(t)$
is an admissible faceted test function such that $\vp$ is
in general position of radius $\eta$
with respect to $u$
at a point $(\hat x, \hat t) \in Q$
then there exists $\de \in (0, \eta)$ such that
\begin{align*}
\vp_t(\hat x, \hat t)
    + F\pth{0, \essinf_{B_\de(\hat x)} \bra{-\partial^0 E(\psi)}} \leq 0.
\end{align*}

\item (\emph{conventional test})
If $\vp \in C^{2,1}_{x,t}(U)$ in a neighborhood $U \subset Q$ of
a point
$(\hat x, \hat t)$,
such that $u - \vp$ has a local maximum at $(\hat x, \hat t)$
and $\abs{\nabla \vp} (\hat x, \hat t) \neq 0$, then
\begin{align*}
\vp_t(\hat x, \hat t)
    + F\pth{\nabla \vp(\hat x, \hat t),
        k(\nabla \vp(\hat x,\hat t),
        \nabla^2 \vp(\hat x,\hat t))} \leq 0,
\end{align*}
where $\nabla^2$ is the Hessian and
\begin{align}
\label{Lambda-nondeg}
    k(p, X) &:= \trace \bra{(\nabla^2 W)(p) X}
        && \text{for $p \in \Rn \setminus \set0$,
            $X \in \mathcal{S}^n$},
\end{align}
so that 
$k(\nabla \vp(\hat x, \hat t), \nabla^2 \vp(\hat x, \hat t))
    = \bra{\divo (\nabla W)(\nabla \psi)}(\hat x, \hat t)$.
Here $\mathcal{S}^n$ is the set of $n \times n$-symmetric
matrices.
\end{enumerate}

A \emph{viscosity supersolution} can be defined similarly as
a lower semi-continuous function, replacing maximum by minimum,
$\leq$ by $\geq$, and $\essinf$ by $\esssup$. Furthermore, in (i)
$\vp$ must be such that $-\vp$ is in a general position of radius $\eta$
with respect to $-u$
(see also Remark~\ref{re:support-function-symmetry}).

Function $u$ is a \emph{viscosity solution}
if it is both a subsolution and supersolution.
\end{definition}

The next result indicates that it is possible to
find an admissible test function in general position
for a given upper semi-continuous function $u$
given an admissible facet that is in general position
with respect to the facet of $u$.

\begin{lemma}
\label{le:faceted-construction}
Suppose that $(H_-, H_+) \in \mathcal P$
is an admissible pair,
and let $u \in USC(Q)$ be a bounded upper semi-continuous
function on $Q := \Tn \times (0,T)$ for some $T >0$,
and let $g \in C^1(\R)$.
Moreover,
let $(\hat x,\hat t) \in Q$
be a point such that $\hat x \in \Tn \setminus \cl{H_- \cup H_+}$.
Suppose that there is $\de > 0$
such that
\begin{align*}
\pair(u(\cdot, t) - u(\hat x, \hat t) - g(t))
\preceq \nbd^{-\delta}(H_-, H_+)
\qquad \text{for $t \in (\hat t - \delta, \hat t + \delta)$.}
\end{align*}
Then there exists a support function $\psi \in \dom(\partial E)$
of $(H_-, H_+)$
and $\eta > 0$
such that
$\vp(x,t) = \psi(x) + g(t)$
is an admissible faceted test function
at $(\hat x, \hat t)$ with pair $(H_-, H_+)$
in a general position of radius $\eta$
with respect to $u$ at a point $(\hat x, \hat t)$.
\end{lemma}

\begin{proof}
Let us first set
\begin{align*}
\eta := \ov2 \min \set{\de, \dist(\hat x, \cl{H_+ \cup H_-})}.
\end{align*}
Then we introduce the function $\ta$ by
\begin{align*}
\ta(x) :=
\sup_{h \in \cl B_\eta(0)} \sup_{t \in [\hat t -\eta, \hat t + \eta]}
u(x + h, t) - u(\hat x, \hat t) - g(t).
\end{align*}
Clearly $\ta\in USC(\Tn)$.
Observe that, since $\cl B_\eta(\hat x) \in \Tn \setminus \cl{H_+ \cup H_-}$
by the definition of $\eta$,
the function
$\vp(x,t) = \psi(x) + g(t)$
is in general position of radius $\eta$ with respect to $u$ at the point
$(\hat x, \hat t)$
if and only if
\begin{align*}
\ta \leq \psi \qquad \text{on $\Tn$.}
\end{align*}
But such a function $\psi \in \dom(\partial E)$ is provided
by Lemma~\ref{le:support-construction}.
\end{proof}

\section{Comparison principle}
\label{sec:comparison-principle}

In this section we will establish the comparison principle
for viscosity solutions introduced in Definition~\ref{def:visc-sol}.
We will fix the spacetime cylinder $Q := \Tn \times (0,T)$.

\begin{theorem}[Comparison]
\label{th:comparison}
Let $u$ and $v$ be respectively a bounded viscosity subsolution
and a viscosity supersolution of \eqref{tvf}
on $Q$.
If $u \leq v$ at $t = 0$ then $u \leq v$ on $Q$.
\end{theorem}

We shall give a slightly different exposition
of the proof of the theorem than the one that appears in \cite{GGP13},
but the method is identical.

We perform a variation of the doubling-of-variables 
procedure:
we define
\begin{align*}
w(x,t,y,s) := u(x,t) - v(y,s),
\end{align*}
and, for a positive constant $\e > 0$
and point $\zeta \in \Tn$,
the functions
\begin{align*}
\Psi_\zeta(x,t,y,s; \e) &:= 
    \frac{\abs{x - y -\zeta}^2}{2\e} + S(t,s; \e),\\
S(t,s; \e) &:= 
    \frac{\abs{t -s}^2}{2\e} + \frac{\e}{T - t}
    + \frac{\e}{T - s},
\end{align*}
where $\abs{x - y - \zeta}$ was defined
in \eqref{torus-metric}.

We analyze the maxima of functions
\begin{align*}
\Phi_\zeta(x,t,y,s;\e)
    := w(x,t,y,s) - \Psi_\zeta(x,t,y,s;\e)
    \qquad \text{for $\zeta \in \Tn$.}
\end{align*}
Following \cite{GG98ARMA}, we define the maximum of $\Phi_\zeta$
\begin{align*}
\ell(\zeta;\e) = \max_{\cl Q \times \cl Q} \Phi_\zeta(\cdot; \e)
\end{align*}
and
the sets of points of maximum of $\Phi_\zeta$,
over $\cl Q \times \cl Q$
\begin{align*}
\mathcal A(\zeta; \e) := \argmax_{\cl Q \times \cl Q} \Phi_\zeta(\cdot;\e)
    := \set{(x,t,y,s) \in \cl Q \times \cl Q :
    \Phi_\zeta(x,t,y,s;\e) = \ell(\zeta;\e)}.
\end{align*}

Suppose that the comparison principle, Theorem~\ref{th:comparison},
does not hold,
that is, suppose that
\begin{align*}
m_0 := \sup_Q \bra{u - v} > 0.
\end{align*}
We have the following proposition.

\begin{proposition}
\label{pr:maxima-interior}
There exists $\e_0 > 0$ such that
for all $\e \in (0, \e_0)$
we have
\begin{align*}
\mathcal A(\zeta;\e)\subset Q \times Q
\qquad \text{for all $\abs{\zeta} \leq \kappa(\e)$},
\end{align*}
where $\kappa(\e) := \ov8 (m_0 \e)^{\ov2}$.
Moreover,
\begin{align*}
\abs{x - y - \zeta} \leq (M\e)^{\ov2},\qquad
\abs{t - s} \leq (M \e)^{\ov2}, \qquad
\text{for all $(x,t,y,s) \in \mathcal A(\zeta; \e)$,}
\end{align*}
where $M := \sup_{\cl Q \times \cl Q} w < \infty$.
\end{proposition}

\begin{proof}
See \cite[Proposition 7.1, Remark 7.2]{GG98ARMA}.
\end{proof}

In the view of Proposition~\ref{pr:maxima-interior},
we fix one $\e \in (0,\e_0)$
such that $(M \e)^{\ov2} < \ov4$
for the rest of the proof
and drop the dependence of the formulas below on $\e$
for the sake of clarity.
Moreover, we introduce
\begin{align*}
\la := \frac{\kappa(\e)}2.
\end{align*}

Again, following \cite{GG98ARMA},
we define
the set of gradients
\begin{align*}
\mathcal B(\zeta) :=
    \set{\frac{x - y -\zeta}{\e} : (x,t,y,s) \in \mathcal A(\zeta)}
    \subset \Rn,
\end{align*}
where $x - y -\zeta$ is interpreted as a vector in
$(-\ov4,\ov4)^n \subset \Rn$.

The situation can be divided into two cases:
\begin{enumerate}[{Case }I.]
\item
$\mathcal B(\zeta) = \set0$ for all $\abs{\zeta} \leq \kappa(\e)$.
\item
There exists $\zeta \in \Tn$
and $p \in \mathcal B(\zeta)$
such that $\abs{\zeta} \leq \kappa(\e)$
and $p \neq 0$.
\end{enumerate}

\subsection{Case I}

This is the less standard case
since it is necessary to construct
admissible faceted test functions for the faceted test
in the definition of viscosity solutions.
We have $\mathcal B(\zeta) = \set0$ for all $\abs{\zeta} \leq \kappa(\e)$.
In this case,
we apply 
the constancy lemma that was presented in \cite[Lemma 7.5]{GG98ARMA}.

\begin{lemma}[Constancy lemma]
\label{le:constancy}
Let $K$ be a compact set in $\R^N$ for some $N > 1$
and let $h$ be a real-valued upper semi-continuous function on $K$.
Let $\phi$ be a $C^2$ function on $\R^d$ with $1 \leq d < N$.
Let $G$ be a bounded domain in $\R^d$.
For each $\zeta \in G$ assume that
there is a maximizer
$(r_\zeta, \rho_\zeta)\in K$ of
\begin{align*}
    H_\zeta(r, \rho) = h(r, \rho) - \phi(r - \zeta)
\end{align*}
over $K$ such that $\nabla \phi(r_\zeta - \zeta) = 0$.
Then,
\begin{align*}
    h_\phi(\zeta) = \sup\set{H_\zeta(r,\rho): (r, \rho)\in K}
\end{align*}
is constant on $G$.
\end{lemma}

We apply Lemma~\ref{le:constancy}
with the following parameters:
\begin{gather*}
N = 2n + 2, \quad
d= n, \quad
\rho = (y, t, s) \in \Tn \times \R \times \R,\\
K = 
    \set{(x - y, y, t, s):
       (x,y) \in \Tn \times \Tn,
       (t,s) \in [0,T] \times [0,T]},\\
G = B_{2\la}(0),\\
h(r,\rho) = w(r + y, t, y, s) - S(t,s),
\quad
\phi(r) = \frac{\abs{r}^2}{2\e}.
\end{gather*}
$K$ can be treated as a compact subset of $\Rn$
in a straightforward way.
We infer that $\ell(\zeta) = h_\phi(\zeta)$ is constant
for $\abs{\zeta} \leq \la$.

Therefore we have also an ordering
analogous to \cite[Corollary 7.9]{GG98ARMA},
which yields the crucial estimate.

\begin{lemma}
\label{le:orderInDoubling}
Let $(\hat x, \hat t, \hat x, \hat s) \in \mathcal A(0)$. Then
\begin{align*}
u(x,t) - v(y,s) - S(t,s)
    \leq u(\hat x, \hat t) - v(\hat x, \hat s) - S(\hat t, \hat s)
\end{align*}
for all $s,t \in (0,T)$
and $x,y \in \Tn$ such that $\abs{x -y} \leq \la := \kappa(\e)/2$.
\end{lemma}

From now on, we fix $(\hat x, \hat t, \hat x, \hat s) \in \mathcal A(0)$
and set
\begin{align*}
\alpha := u(\hat x, \hat t), \qquad \beta := v(\hat x, \hat s).
\end{align*}
As in \cite{GGP13},
we introduce the closed sets
\begin{align*}
U := \set{x : u(x, \hat t) \geq \al},\qquad
V := \set{x : v(x, \hat s) \leq \be},
\end{align*}
which will be used to generate strictly ordered
smooth facets.
To accomplish that,
let us now for simplicity set $r := \la/10$.
Furthermore, define the closed sets
\begin{align*}
X := \cl{(\nbd^r(U))^c},\qquad
Y := \cl{(\nbd^r(V))^c}.
\end{align*}
Since $\dist(U, X) = \dist(V, Y) = r$,
the definition of $U$ and $V$ and the semi-continuity
of $u$ and $v$ imply that there exists $\de > 0$
such that
\begin{subequations}
\label{bound-XY}
\begin{align}
u(x,t) - \al  + S(\hat t, \hat s) - S(t,\hat s) &< 0,&
x&\in X, t \in [\hat  t - \de, \hat t + \de],\\
v(x,t) - \be + S(\hat t, t) - S(\hat t, \hat s) &> 0,&
x&\in Y, t \in [\hat  s - \de, \hat s + \de].
\end{align}
\end{subequations}
Moreover, the estimate from Lemma~\ref{le:orderInDoubling}
and the definition of $U$ and $V$
implies that
\begin{subequations}
\label{bound-UV}
\begin{align}
u(x,t) - \al  + S(\hat t, \hat s) - S(t,\hat s) &\leq 0,&
x&\in \nbd^\la(V), t \in (0,T),\\
v(x,t) - \be + S(\hat t, t) - S(\hat t, \hat s) &\geq 0,&
x&\in \nbd^\la(U), t \in (0,T).
\end{align}
\end{subequations}

This suggests introducing
\begin{align*}
g_u(t) := S(t, \hat s) - S(\hat t, \hat s),\qquad
g_v(t) := S(\hat t, \hat s) - S(\hat t, t)
\end{align*}
and the pairs
\begin{align*}
P_u(t) &:= \pair(u(\cdot, t) - \al - g_u(t)),\\
P_v(t) &:= \pair(v(\cdot, t) - \be - g_v(t)).
\end{align*}
If we denote by $-P_v(t)$ the reversed pair of $P_v(t)$,
we infer from \eqref{bound-XY} and \eqref{bound-UV} in particular that
\begin{subequations}
\label{def-Ru-Rv}
\begin{align}
P_u(t) &\preceq ((\nbd^r(U))^c, \nbd^r(U) \setminus \nbd^\la(V)) =: R_u,\\
-P_v(t) &\preceq ((\nbd^r(V))^c, \nbd^r(V) \setminus \nbd^\la(U)) =: R_v.
\end{align}
\end{subequations}

We define the pairs
\begin{align*}
S_u := (U^c, U \setminus \nbd^{\la-3r}(V)),
\qquad
S_v:= (V^c, V \setminus\nbd^{\la-3r}(U)).
\end{align*}
Since $S_u, S_v \in \mathcal P$,
Proposition~\ref{pr:smooth-pair-approx}
implies that there exist smooth pairs $(U_-, U_+)$ and
$(V_-, V_+)$
such that
\begin{subequations}
\label{facet-est}
\begin{align}
\label{Pu-facet}
\nbd^{2r}(S_u) &\preceq
(U_-, U_+) \preceq
\nbd^{3r}(S_u),\\
\label{Pv-facet}
\nbd^{2r}(S_v) &\preceq
(V_-, V_+) \preceq
\nbd^{3r}(S_v).
\end{align}
\end{subequations}

Before proving Lemma~\ref{le:approx-facet-properties} below,
we give the following trivial estimate.

\begin{lemma}
\label{le:split-estimate}
Suppose that $G, H \subset \Tn$.
Then
\begin{align*}
\nbd^\rho(G) \setminus \nbd^\rho(H)
\subset \nbd^\rho(G\setminus H)\qquad\text{for any $\rho > 0$.}
\end{align*}
\end{lemma}

\begin{proof}
Suppose that $x \in \nbd^\rho(G) \setminus \nbd^\rho(H)$.
Then there exists $y \in G$ such that $x \in \cl B_\rho(y)$.
In particular, $y \notin H$
and therefore $y \in G \setminus H$,
which implies that $x \in \nbd^\rho(G\setminus H)$.
\end{proof}

\begin{lemma}
\label{le:approx-facet-properties}
The pair  $(U_-, U_+)$ and
the pair
$(V_-, V_+)$
have the following properties:
\begin{enumerate}
\item The pairs are strictly ordered in the sense
\begin{align}
\label{strict-ord}
\nbd^{r}(U_-, U_+) \preceq (V_+, V_-) = -(V_-, V_+).
\end{align}
\item The contact point $\hat x$ lies 
in the interior of the intersection of the facets, that is,
\begin{align}
\label{in-facet}
\cl B_r(\hat x) \subset U_-^c \cap U_+^c \cap V_-^c \cap V_+^c.
\end{align}

\item The pairs are in general position with respect to $R_u$ and $R_v$,
i.e.,
\begin{align}
\label{strict-u-v}
\nbd^r(R_u)\preceq (U_-, U_+), \qquad
\nbd^r(R_v) \preceq (V_-, V_+).
\end{align}
\end{enumerate}
\end{lemma}
\begin{proof}
Let us recall the properties of $\nbd^\rho$
in Proposition~\ref{pr:nbd-properties}.
To show (a),
first estimate using \eqref{Pu-facet}
and \eqref{compl-nbd}
\begin{align}
\label{U-est}
\nbd^r(U_+) \subset
\nbd^r (\nbd^{3r}(U \setminus \nbd^{\la-3r}(V)))
\subset \nbd^{3r}(\nbd^{3r-\la}(V^c)) \subset \nbd^{6r - \la}(V^c).
\end{align}
On the other hand, since $6r - \la = -4r < -3r$,
we have from \eqref{Pv-facet}
\begin{align*}
\nbd^{6r - \la}(V^c) \subset \nbd^{-3r}(V^c) \subset V_-.
\end{align*}
Combining these two estimates we get $\nbd^r(U_+) \subset V_-$.
Symmetric estimates show that $\nbd^r(V_+) \subset U_-$,
i.e., $V_+ \subset \nbd^{-r}(U_-)$.
Consequently, \eqref{strict-ord} follows.

For (b), we first realize that by definition
$\hat x \in U \cap V$
and therefore $\cl B_r(\hat x) \subset \nbd^r(U) \cap \nbd^r(V)$.
But \eqref{Pu-facet}
yields
\begin{align*}
U_- \subset \nbd^{-2r} (U^c) = (\nbd^{2r}(U))^c.
\end{align*}
Similarly, \eqref{U-est} implies
\begin{align*}
U_+ \subset \nbd^{6r - \la}(V^c) \subset \nbd^{-2r}(V^c) =
(\nbd^{2r}(V))^c.
\end{align*}
Symmetric estimates hold for $V_\pm$ and
(b) follows.

To show (c), we estimate using Lemma~\ref{le:split-estimate}
\begin{align*}
\nbd^r(R_u) &\preceq
(\nbd^{-2r}(U^c), \nbd^{2r}(U) \setminus \nbd^{\la -r}(V))
\\&\preceq
(\nbd^{-2r}(U^c), \nbd^{2r}(U  \setminus \nbd^{\la -3r}(V)))
= \nbd^{2r}(S_u) \preceq (U_-, U_+).
\end{align*}
The statement for $(V_-, V_+)$ is analogous.
\end{proof}

We can finally finish the construction
for Case II
using the estimates in Lemma~\ref{le:approx-facet-properties}.
Indeed, the estimate \eqref{strict-u-v},
recalling the definitions of $R_u$ and $R_v$ in \eqref{def-Ru-Rv},
is all that is necessary to apply Lemma~\ref{le:faceted-construction},
with an obvious modification for $v$.
Then we have $\vp_u(x,t) = \psi_u(x) + g_u(t)$
(resp. $\vp_v(x,t) = \psi_v(x) + g_v(t)$),
an admissible faceted test function
at $(\hat x, \hat t)$  (resp. $(\hat x, \hat s)$)
with facet $(U_-, U_+)$ (resp. $(V_+, V_-) = -(V_-, V_+)$).
Moreover,
$\vp_u$ is in general position with respect to $u$
at $(\hat x, \hat t)$
and $-\vp_v$ is in general position with respect to $-v$
at $(\hat x, \hat s)$, both with some radius $\eta > 0$.
Since the facets are strictly ordered \eqref{strict-ord},
the monotonicity Proposition~\ref{pr:monotonicity}
and \eqref{in-facet}
yield
\begin{align}
\label{ess-order}
\essinf_{\cl B_\eta(\hat x)} \bra{-\partial^0 E(\psi_u)}
\leq \esssup_{\cl B_\eta(\hat x)} \bra{-\partial^0 E(\psi_v)}.
\end{align}
By definition of viscosity solutions, we have
\begin{align*}
(g_u)_t(\hat t)
+ F\pth{0, \essinf_{\cl B_\eta(\hat x)} \bra{-\partial^0 E(\psi_u)}} &\leq 0,\\
(g_v)_t(\hat s)
+ F\pth{0, \esssup_{\cl B_\eta(\hat x)} \bra{-\partial^0 E(\psi_v)}} &\geq 0.
\end{align*}
Subtracting these two inequalities
and using \eqref{ess-order} with the ellipticity of $F$ \eqref{ellipticity},
we arrive at
\begin{align*}
0 < \frac{\e}{(T- \hat t)^2} +\frac{\e}{(T- \hat s)^2}
&+ F\pth{0, \essinf_{\cl B_\eta(\hat x)} \bra{-\partial^0 E(\psi_u)}}
\\&- F\pth{0, \esssup_{\cl B_\eta(\hat x)} \bra{-\partial^0 E(\psi_v)}}
\leq 0,
\end{align*}
a contradiction.
Therefore we conclude that Case I cannot occur.

\subsection{Case II}
This the more classical case since there exists $\zeta \in \Tn$
and $p \in \mathcal B(\zeta)$
such that $\abs{\zeta} \leq \kappa(\e)$
and $p \neq 0$,
and we only need to construct a smooth test function
for the classical test in the definition of viscosity solutions.
Here we refer the reader to \cite{GG98ARMA,GGP13}.
We again arrive at a contradiction,
yielding that Case II cannot occur either.
Therefore the comparison principle Theorem~\ref{th:comparison}
holds.

\section{Existence of solutions via stability}
\label{sec:existence-stability}

\subsection{Stability}
\label{sec:stability}

In this section we discuss the stability
of solution of \eqref{tvf}
under an approximation by regularized
problems.

Suppose that $\set{W_m}_{m\in\N}$
is a decreasing sequence of
$C^2$ functions on $\Rn$
that converge locally uniformly to $W$
and such that the functions $W_m$
satisfy
\begin{align*}
a_m^{-1}I \leq \nabla^2 W_m(p) \leq a_m I
\qquad \text{for all $p \in \Rn$, $m\in\N$}
\end{align*}
and some sequence of positive numbers $a_m$.

\begin{example}
Let $\phi_{\ov m}$ be the standard mollifier
with support of radius $\ov m$.
Define the smoothing
\begin{align*}
W_m(p) = (W * \phi_{\ov m})(p) + \ov m \abs{p}^2 \qquad p \in \Rn.
\end{align*}
By convexity we have $W_m > W$
(clearly true for $p \neq 0$, and at $p = 0$
we use \eqref{W-bound}),
$W_m \in C^\infty$,
$\nabla^2 W_m \geq \ov m I$
and $W_\e \downarrow W$ as $\e\to0$ locally uniformly.
The bound on $\nabla^2 W_m$ from above follows from the one-homogeneity of
$W$ which yields $\nabla^2 W(a p) = a^{-1} \nabla^2 W(p)$ for $a>0$.
\end{example}

Let us introduce the regularized energies
\begin{align*}
E_m(\psi) :=
\begin{cases}
\int_\Tn W_m(\nabla \psi) & \psi \in H^2(\Tn),\\
+\infty & \psi \in L^2(\Tn) \setminus H^2(\Tn),
\end{cases}
\end{align*}
where $H^2(\Tn)$ is the standard Sobolev space.

We shall approximate the problem \eqref{tvf}
by a sequence of problems
\begin{align}
\label{approximate-problem}
u_t + F(\nabla u, -\partial^0 E_m(u(\cdot, t))) = 0,
\intertext{with initial data}
\at{u}{t=0} = u_0.
\end{align}

We have the following proposition proved
in \cite{GGP13}.

\begin{proposition}
\label{pr:Em-properties}
\ \begin{enumerate}[(a)]
\item
$E_m$ form a decreasing sequence of
proper convex lower semi-continuous functionals
on $L^2(\Tn)$
and
$E = \pth{\inf_m E_m}_*$,
the lower semi-continuous envelope of $\inf_m E_m$ in $L^2(\Tn)$.

\item
The subdifferential $\partial E_m$ 
is a singleton for all $\psi \in \mathcal{D}(\partial E_m) = H^2(\T^n)$
and its canonical restriction can be expressed as
\begin{align}
\label{approximate-operator}
    -\partial^0 E_m(\psi)
    = \divo \bra{(\nabla W_m)(\nabla \psi)}
    = \trace \bra{(\nabla^2 W_m)(\nabla \psi) \nabla^2 \psi}
        \quad \text{a.e.}
\end{align}

\item
Due to the ellipticity of $F$,
the problem \eqref{approximate-problem} is a degenerate parabolic
problem that has a unique global viscosity solution
for given continuous initial data $u_0 \in C(\T^n)$.
\end{enumerate}
\end{proposition}

The main theorem of this section
is the stability of solutions of \eqref{tvf}
with respect to the half-relaxed limits
\begin{align*}
\halflimsup_{m\to\infty} u_m(x,t) &:= 
    \lim_{k\to\infty} \sup_{m \geq k} \sup_{\abs{y - x} \leq \ov k}
    \sup_{\abs{s - t} \leq \ov k} u_m(y,s), \\ 
    \halfliminf_{m\to\infty} u_m(x,t) &:= 
    -\halflimsup_{m\to\infty} (-u_m)(x,t).
\end{align*}

\begin{theorem}[Stability]
\label{th:stability}
Let $u_m$ be a sequence of viscosity 
subsolutions of \eqref{approximate-problem}
on $\Tn \times [0,\infty)$,
and let $\overline u = \halflimsup_{m\to\infty} u_m$.
Assume that $\overline u < +\infty$ in $\Tn \times [0,\infty)$.
Then $\overline u$ is a viscosity subsolution of \eqref{tvf}.

Similarly, $\underline u = \halfliminf_{m\to\infty} u_m$
is a viscosity supersolution of \eqref{tvf}
provided  that $u_m$ is a sequence of viscosity supersolutions of
\eqref{approximate-problem} and $\underline u > -\infty$.
\end{theorem}

The proof is the same as in \cite{GGP13}
and we shall skip it here.

\subsection{Existence}

We shall use the stability theorem to prove the following
existence result.

\begin{theorem}[Existence]
\label{th:existence}
If $F$ is continuous and degenerate elliptic \eqref{ellipticity},
and $W$ satisfies \eqref{W-regularity} and \eqref{W-bound},
and $u_0 \in C(\Tn)$,
there exists a unique solution
$u \in C(\Tn \times [0, \infty))$ of \eqref{tvf}
with the initial data $u_0$.
Furthermore, if $u_0 \in \Lip(\Tn)$ then
\begin{align*}
\norm{\nabla u(\cdot, t)}_\infty \leq \norm{\nabla u_0}_\infty.
\end{align*}
\end{theorem}

\newcommand{\ou}{{\overline u}}
\newcommand{\uu}{{\underline u}}

The proof of the theorem will proceed in three steps:
\begin{inparaenum}[1)]
\item due to the stability, by finding the solution $u_m$ of
the problem \eqref{approximate-problem}
for all $m \geq 1$,
we can find a subsolution $\ou$ and a supersolution $\uu$
of \eqref{tvf};
\item \label{existence-step-2}
    a barrier argument at $t = 0$ shows that
$\ou$ and $\uu$ have the correct initial data $u_0$; and
\item \label{existence-step-3}
the comparison principle shows that $\ou = \uu$
is the unique viscosity solution of \eqref{tvf},
and the Lipschitz estimate holds.
\end{inparaenum}

Before giving a proof of the existence theorem,
we construct barriers for step \ref{existence-step-2}.

Since the operator \eqref{approximate-operator}
degenerates at points where $\nabla u = 0$ as $m \to\infty$,
it seems to be necessary to construct barriers that depend
on $m$.
We will use the Wulff functions for energy $E_m$;
these were previously considered in the proof of
stability for general equations of the type \eqref{approximate-problem}
in one-dimensional setting in \cite{GG99CPDE}
and in the isotropic setting in \cite{GGP13}.
However, the construction is slightly more complicated in
the anisotropic case in higher dimension.
Since the operator $F$ in \eqref{tvf} depends on the derivative
of the solutions, we have to construct test functions that
have uniformly bounded space derivatives as $m \to \infty$.
However, the derivatives of the Wulff functions for $E_m$
blow up as $m \to \infty$.
Therefore we have to cut off large derivatives.
This was done in \cite{GG99CPDE, GGP13} by a simple idea
that can be only applied for one-dimensional
or radially symmetric Wulff functions.
Here we present a different idea that relies on the modification of $W_m$
directly using the properties of the Legendre-Fenchel transform.

For a convex proper function $\phi : \Rn \to (-\infty, +\infty]$
we define its convex conjugate $\phi^\star$
via the Legendre-Fenchel transform as
\begin{align*}
\phi^\star(x) := \sup_{p \in \Rn} \bra{x \cdot p - \phi(p)}.
\end{align*}
It is well-known that $\phi^\star$ is also convex
and that, if $\phi$ is also lower semi-continuous,
$\phi^{\star\star} = \phi$.

We give the proof of the following lemma for completeness.

\begin{lemma}
\label{le:bounded-conj}
Let $\Omega \subset \Rn$ be a non-empty bounded convex open set
and let $\phi \in LSC(\Rn)$ be a convex function on $\Rn$
such that $\phi \in C^2(\Omega)$, $\phi = \infty$ on $\Rn \setminus \Omega$
and $\phi$ is strictly convex in $\Omega$, i.e., $\nabla^2 \phi > 0$
in $\Omega$.

Then $\phi^\star \in C^2(\Rn) \cap \Lip(\Rn)$,
\begin{align}
\label{2nd-der}
\nabla \phi^\star(x) \in \Omega \quad \text{and} \quad
\nabla^2 \phi^\star(x) =
\bra{\nabla^2 \phi(\nabla \phi^\star(x))}^{-1} > 0
\qquad x \in \Rn.
\end{align}
\end{lemma}

\begin{proof}
Since $x \cdot p - \phi(p)$ is upper semi-continuous
and $x \cdot p - \phi(p) = -\infty$ on $\partial \Omega$,
the supremum in the definition is for every $x \in \Rn$
attained at a point $p \in \Omega$
such that $\nabla \phi(p) = x$.
Additionally, $p$ is unique due to the strict convexity,
and the function $p(x)$ is $C^1$ by the inverse function theorem.
If we differentiate $\phi^\star(x) = x \cdot p(x) - W(p(x))$
we get $\nabla \phi^\star(x) = p(x) \in \Omega$.
Thus $\nabla \phi^\star$ is the inverse map of $\nabla \phi$
and the inverse function theorem implies
the expression for $\nabla^2 \phi^\star(x)$.
\end{proof}

Let $\psi:\Rn \to (-\infty, \infty]$ be a
lower semi-continuous convex
function such that $\psi \in C^\infty(B_1(0))$
and
$\psi(p) = \infty$ for $\abs{p} \geq 1$
and $\psi(0) = 0$.
Note that the semi-continuity implies that
$\psi(p) \to \infty$ as $\abs{p} \to 1^-$.
For given positive constants $m, A, q$, we define
\begin{align*}
W_{m;A,q}(p) := A \pth{W_m(p)
            + q \psi\pth{\frac p q} - W_m(0)}.
\end{align*}
We also define the quasilinear differential operators
$\mathcal L_m : C^2(\Rn) \to \R$ for $m \in \N$
motivated by the expression for $-\partial^0 E_m$
in \eqref{approximate-operator}
as
\begin{align*}
\mathcal L_m (u)(x) := \trace \bra{(\nabla^2 W_m)(\nabla u(x)) \nabla^2 u(x)}
\qquad u \in C^2(\Rn).
\end{align*}

Functions $W_{m;A,q}^*$, the conjugates of $W_{m;A,q}$,
approximate the Wulff functions $W_m^*$ of the energies $E_m$
and we summarize their properties in the following lemma.

\begin{lemma}
\label{le:conj-estimate}
For any $m, A, q$ positive, $W_{m;A,q}^*$ are
strictly convex, nonnegative, $C^2$ functions on $\Rn$
and
\begin{align*}
\abs{\nabla W_{m;A,q}^\star(x)} < q, \qquad
0 < \mathcal L_m(W_{m;A,q}^\star)(x) \leq A^{-1} n \qquad x \in \Rn.
\end{align*}
\end{lemma}

\begin{proof}
Strict convexity and regularity follows from Lemma~\ref{le:bounded-conj}.
In particular, we observe that $\Omega = B_q(0)$
and hence $\nabla W_{m;A,q}^\star \in B_q(0)$.
Nonnegativity is also obvious.

Let $x \in \Rn$ and set $p = \nabla W_{m;A,q}^\star(x)$.
Since $\psi$ in the definition of $W_{m;A,q}$ is convex
and thus $\nabla^2 \psi \geq 0$ on $B_1(0)$,
\eqref{2nd-der}
yields
\begin{align*}
0 < \nabla^2 W_{m;A,q}^\star(x) &= (\nabla^2 W_{m;A,q}(p))^{-1}\\
&= A^{-1}\bra{\nabla^2 W_m(p)
+ \frac1 q (\nabla^2 \psi)\pth{\frac{p}{q}}}^{-1}\\
&\leq A^{-1}\bra{\nabla^2 W_m(p)}^{-1}.
\end{align*}
We also recall that if $M, N \geq 0$ then also $\trace MN \geq 0$.
Therefore
\begin{align*}
\mathcal L_m(W_{m;A,q}^\star)(x) =
\trace \bra{(\nabla^2 W_m)(p) \nabla^2 W_{m;A,q}^\star(x)}\leq
A^{-1} \trace I = A^{-1} n.
\end{align*}
Similarly $\mathcal L_m(W_{m;A,q}^\star)(x) > 0$.
\end{proof}

Now we define the barriers
\begin{align*}
\overline\phi_{m;A,q}(x,t) &:= \be_{A,q} t + W_{m; A,q}^\star(x),\\
\underline\phi_{m;A,q}(x,t) &:= -\be_{A,q} t - W_{m; A,q}^\star(-x),
\end{align*}
where
\begin{align}
\label{be-Aq}
\be_{A,q} := \sup_{p \in B_q(0)} \sup_{\abs{\xi} \leq A^{-1} n} \abs{F(p, \xi)} + 1< \infty.
\end{align}

\begin{corollary}
For any $m, A, q > 0$
the function $\overline \phi_{m;A,q}$ is a classical supersolution
of \eqref{approximate-problem} on $\Rn$
and the function $\underline \phi_{m;A,q}$ is a classical subsolution
of \eqref{approximate-problem} on $\Rn$.
\end{corollary}

\begin{proof}
The corollary follows from Lemma~\ref{le:conj-estimate}
and the definition of $\be_{A,q}$ in \eqref{be-Aq}.
Additionally, we observe that if $u \in C^2(\Rn)$
and $v(x) = - u(-x)$ then
\begin{align*}
\mathcal L_m(v)(x) = -\mathcal L_m(u)(-x).
\end{align*}
\end{proof}

Finally, we observe that $W_{m; A, q}^\star$ can be bound from below
away from the origin.

\begin{lemma}
\label{le:lower-bound-W}
For any $\de, K > 0$ there exist $m_0, A, q > 0$
such that
\begin{align*}
W_{m;A, q}^\star (x) \geq 2K \qquad \text{for all $x$, $\abs{x} \geq \de$,
and $m \geq m_0$.}
\end{align*}
\end{lemma}

\begin{proof}
Let us define
\begin{align}
\label{mu-def}
\mu := \sup_{\abs{p}=1/2} \bra{W(p) + \psi(p)} \in (0, \infty).
\end{align}
Now we set
\begin{align*}
A := \frac\de{8\mu}, \qquad q := \frac{8 K}{\de}.
\end{align*}
By the locally uniform convergence of $W_m \to W$,
we can find $m_0 > 0$ such that
\begin{align*}
\sup_{\abs{p} = q/2} \abs{W_m(p) -W_m(0) - W(p)} \leq q\mu
\qquad m > m_0.
\end{align*}
Now for any $x$ such that $\abs{x} \geq \de$
and any $m > m_0$,
setting $p = \frac{q}{2} \frac{x}{\abs{x}}$,
we estimate
\begin{align*}
W_{m;A,q}^\star(x) &\geq x \cdot p
- W_{m;A,q}\pth{p}\\
&= \frac{q}{2}\abs{x}
- A \pth{W_m(p)
            + q \psi\pth{\frac pq} - W_m(0)}\\
&\geq \frac{q}{2}\abs{x}
- A \pth{W(p)
            + q \psi\pth{\frac p q} + q\mu}\\
&= \frac q2 \abs{x} - A q \pth{W\pth{\frac pq}
    +\psi\pth{\frac p q} + \mu}\\
&\geq \frac q2 \abs{x} - 2Aq\mu \geq 2K,
\end{align*}
where we used the one-homogeneity \eqref{W-bound} of $W$
and \eqref{mu-def}.
\end{proof}

With the constructed barriers, we are able to finish the proof of
the existence theorem.

\begin{proof}[Proof of Theorem~\ref{th:existence}]
Let $W_m$ be a sequence that approximates $W$
as in Section~\ref{sec:stability}.
By Proposition~\ref{pr:Em-properties},
the approximate problem \eqref{approximate-problem}
with initial data $u_0$
has a unique continuous solution $u_m$ on $\Tn \times [0,\infty)$.
Since function $(x,t) \mapsto F(0,0) t + \al$ is a solution
of \eqref{approximate-problem} for any $m$ and $\al \in \R$,
$u_m$ are locally uniformly bounded by the comparison principle.
Therefore the stability result,
Theorem~\ref{th:stability},
yields that
$\ou = \halflimsup_{m\to\infty} u_m$
is a subsolution of \eqref{tvf}
and $\uu = \halfliminf_{m\to\infty} u_m$
is a supersolution of \eqref{tvf}.
Clearly $\uu \leq \ou$.

We are left to prove that $\ou(x,0) \leq u_0 \leq \uu(x,0)$
since then the comparison principle,
Theorem~\ref{th:comparison},
yields that $\ou = \uu $ on $\Tn \times [0,\infty)$
and $\ou = \uu$ is the unique solution of \eqref{tvf}
with initial data $u_0$.

Let us thus set $K := \sup_\Tn \abs{u_0} < \infty$
and choose $\xi \in \Tn$ and $\e > 0$.
We shall show that $\ou(\xi, 0) \leq u_0(\xi) + 2\e$.
By continuity, there exists $\de>0$ such that
$u_0(x) \leq u_0(\xi) +\e$
for $x \in B_\de(\xi)$.
Let $m_0, A$ and $q$ be the constants given by
Lemma~\ref{le:lower-bound-W} and define
\begin{align*}
\phi_m(x,t) := \inf_{k \in \Z^n} \overline\phi_{m;A,q}(x + k - \xi, t)
+ u_0(\xi) + \e.
\end{align*}
Observe that $\phi_m$ is a viscosity supersolution of
\eqref{approximate-problem} for every $m$.
Moreover, by Lemma~\ref{le:lower-bound-W} and \eqref{le:conj-estimate},
and the choice of the parameters,
$u_0 \leq \phi_m(\cdot, 0)$ on $\Tn$ for all $m > m_0$.
Therefore the comparison principle yields
$u_m \leq \phi_m$ on $\Tn \times [0, \infty)$.
Finally, it is easy to observe that $\phi_m(\xi, 0) \leq u_0(\xi) + 2\e$
for all sufficiently large $m$.
Since $\phi_m$ are $q$-Lipschitz continuous in space by
Lemma~\ref{le:conj-estimate},
we have
\begin{align*}
u_m(x,t) \leq \phi_m(x,t) \leq \be_{A,q} t + q\abs{x - \xi} + u_0(\xi) + 2\e
\end{align*}
for all large $m$.
Hence $\ou(\xi,0) \leq u_0(\xi) + 2\e$.

Since $\e$ was arbitrary, we conclude that $\ou(\xi, 0) \leq u_0$.
A similar argument with $\underline\phi$
yields $\uu(\xi,0) \geq u_0$.

A standard argument yields the Lipschitz continuity of the solution.
\end{proof}

\parahead{Acknowledgments}
The work of the second author is partly supported by Japan Society for the Promotion of Science through grants
Kiban (S) 21224001,
Kiban (A) 23244015 and Houga 25610025.

\end{document}